\documentclass[11pt]{article}
\usepackage[margin=1.15in]{geometry}
\usepackage{mathptmx}
\usepackage{amsmath,amssymb,amsthm}
\usepackage[colorlinks=true,linkcolor=blue,citecolor=blue,urlcolor=blue]{hyperref}

\theoremstyle{plain}
\newtheorem{theorem}{Theorem}
\newtheorem{proposition}{Proposition}
\newtheorem{lemma}{Lemma}
\newtheorem{corollary}{Corollary}
\theoremstyle{definition}

\newtheorem{remark}{Remark}

\newcommand{\beqn}{\begin{eqnarray*}}
\newcommand{\eeqn}{\end{eqnarray*}}

\newcommand{\Rp}{\mathbb{R}^n_{+}}
\newcommand{\Rpp}{\mathbb{R}^n_{++}}
\newcommand{\xbar}{\bar{x}}

\title{\bf Global stability of Michaelis--Menten networks}
\author{Eduardo D. Sontag%
\thanks{Department of Electrical and Computer Engineering and Department
of Bioengineering, Northeastern University, Boston, Massachusetts 02115,
USA.  Email: \texttt{e.sontag@northeastern.edu}}}
\date{}

\begin{document}
\maketitle

\begin{abstract}
\noindent
In a very nice paper, Blanchini, Breda, Giordano and Liessi
\cite{BBGL} studied networks of species coupled through
Michaelis--Menten interactions.
Under standard assumptions of strong connectedness and non-degeneracy, they show
that there is a unique equilibrium, which is (locally) asymptotically stable.
This note shows that global stability, which was stated there as an
open problem, holds as well.
In fact, whenever a positive equilibrium exists, it attracts every
solution with positive initial condition, with no connectedness or
non-degeneracy hypothesis at all; those assumptions are needed only in
order to guarantee existence and to cover solutions starting on the
boundary of the orthant.
The key is to note that the system is contractive with respect to the
Thompson part metric.
\end{abstract}

\section{Introduction}

The paper by Blanchini, Breda, Giordano and Liessi \cite{BBGL} studied
``Michaelis--Menten'' networks, common in biology, and proved that
these networks have at most one positive equilibrium, which is locally
asymptotically stable.  Furthermore, under strong connectedness and
non-degeneracy, the equilibrium exists.  They left the
question of global stability open, observing that the unweighted
$\infty$-norm of the deviation from equilibrium is not a global
Lyapunov function.  This note starts by discussing a short self-contained
proof of the local results, and then goes on to settle the global
question in the affirmative: under the same hypotheses, the positive
equilibrium attracts the whole nonnegative orthant. The key is to use
ideas from contractive system theory.
A \emph{contractive} system is one for which any two solutions
approach each other exponentially
with a rate that, in the situation studied here, is uniform on compact subsets of the open orthant rather than on the whole orthant. If a system is contractive,
equilibria are unique and globally attracting. The theory of
contractive systems relies upon
matrix measures (also known as logarithmic norms, see for instance
\cite{michelbook,desoer-vidyasagar-siam})
as applied to the
analysis of nonlinear systems (popularized by
Slotine~\cite{LOHMILLER1998683}, see also the self-contained 
expository article~\cite{Sontag2010} and the recent book~\cite{bullo}).

Using non-Euclidean \textit{weighted} metrics has been shown to be the
key for establishing global stability of many types of genetic and
other biological networks, see e.g.~\cite{russo_dibernardo_sontag09,14cdc_aminzare_tutorial}.
The ``right'' metric for Michaelis--Menten networks turns out to be
the logarithmic (Thompson) distance on the positive orthant.
Basically, one uses the weighted Lyapunov function:
\[
V(x)\;=\;\max_i\Bigl|\log\frac{x_i}{\bar x_i}\Bigr|.
\]
It turns out that an inequality already used in \cite{BBGL} only at the
equilibrium is exactly the infinitesimal contraction condition in that
metric, and it holds at every point, not merely equilibria.

\section{The model}

We consider systems of dimension $n\ge 2$, whose state coordinates
(the ``species'') take values (``concentrations'') $x_1,\dots,x_n\ge0$.
The general form of the equations is as follows:
\begin{equation}
\dot x_i \;=\; -\,d_i\,x_i \;+\; f_i(x),
\qquad
f_i(x)\;:=\;b_i+\sum_{j\ne i} h_{ij}(x_j),
\qquad i=1,\dots,n.
\label{eq:sys}
\end{equation}
The constants $d_i$ (``decay rates'') are assumed to be positive,
and the constants $b_i$ (which we may view as constant inputs) are nonnegative.
Each of the interaction functions
$h_{ij}$ is either identically zero or, for some positive constants
$a_{ij}$ and $k_{ij}$, takes one of the following two forms:
\begin{equation}
h_{ij}(s)\;=\;a_{ij}\,\frac{s}{k_{ij}+s}\quad\text{(activation)},
\qquad
h_{ij}(s)\;=\;a_{ij}\,\frac{k_{ij}}{k_{ij}+s}\quad\text{(inhibition)} .
\label{eq:hill}
\end{equation}
Observe that $f_i$ does not depend upon the respective $x_i$, and that
$0\le h_{ij}\le a_{ij}$. 
We think of a graph in which there is an edge $j\to i$ whenever
$h_{ij}\not\equiv 0$ and there is no edge if $h_{ij}\equiv 0$.
Let us write $E:=\{(i,j):j\ne i,\ h_{ij}\not\equiv0\}$ for the set of
edges.  The constants $a_{ij}$ and $k_{ij}$ are defined only for
$(i,j)\in E$, and we make the convention that every sum over $j\ne i$ and every
maximum over pairs $(i,j)$ appearing below is understood to range only
over $E$.
Using this graph terminology, we consider these three properties:
\begin{itemize}\itemsep1pt
\item[(A1)] \emph{no idle node}: for each $i$, either $b_i>0$ or $i$ has at least one incoming edge;
\item[(A2)] \emph{strong connectedness} of the graph;
\item[(A3)] \emph{non-degeneracy}: at least one $h_{ij}$ is inhibitory, or at least one $b_i>0$.
\end{itemize}

For $n\ge2$, (A2) implies (A1). Write $\Rp=[0,\infty)^n$ and $\Rpp=(0,\infty)^n$, and let
\begin{equation}
  \rho(x,z)\;:=\;\max_{1\le i\le n}\Bigl|\log\frac{x_i}{z_i}\Bigr|
  \;=\;
  \max_{1\le i\le n}\Bigl|\log x_i - \log z_i\Bigr|  ,
\qquad x,z\in\Rpp,
\label{eq:rho}
\end{equation}
the \emph{logarithmic distance} on the open orthant. This is Thompson's part metric
\cite{Thompson,LemmensNussbaum}; it is the sup-norm distance between $\log x$ and
$\log z$. It is known that $(\Rpp,\rho)$ is a complete metric space.
Of course, another way to think of this is to change coordinates via
$x \mapsto \log x$, so as to get a system evolving in $\mathbb{R}^n$, and then
considering the usual sup norm in $\mathbb{R}^n$.
%We use the same symbol in dimension one, $\rho(s,\sigma)=|\log (s/\sigma)|$.
Let us also introduce the following map:
\begin{equation}
P(x)\;:=\;\bigl(f_1(x)/d_1,\,\dots,\,f_n(x)/d_n\bigr)
\label{eq:P}
\end{equation}
whose fixed points are the equilibria of~\eqref{eq:sys}.

\section{The key observation}

\begin{lemma}\label{lem:gain}
Let $h$ be either function in \eqref{eq:hill}, with saturation level $a$. Then for every $s>0$
\begin{equation}
\frac{s\,|h'(s)|}{h(s)}\;=\;1-\frac{h(s)}{a}\;\in\;(0,1),
\qquad\text{equivalently}\qquad
h(s)-s\,|h'(s)|\;=\;\frac{h(s)^2}{a}\;>\;0 .
\label{eq:gain}
\end{equation}
Consequently $h$ strictly contracts in logarithmic distance: for all $s,\sigma>0$ with $s\ne\sigma$,
\begin{equation}
\left|\log{h(s)}-\log{h(\sigma)}\right|\;<\;\left|\log{s}-\log{\sigma}\right|
.
%\rho\bigl(h(s),h(\sigma)\bigr)\;<\;\rho(s,\sigma).
\label{eq:contract1d}
\end{equation}
\end{lemma}

\begin{proof}
For activation, $h'(s)=ak/(k+s)^2$ and $h(s)=as/(k+s)$, so
$s h'(s)/h(s)=k/(k+s)=1-h(s)/a$. For inhibition, $h'(s)=-ak/(k+s)^2$ and
$h(s)=ak/(k+s)$, so $s|h'(s)|/h(s)=s/(k+s)=1-h(s)/a$. In both cases
$0<h(s)<a$ for $s>0$, so the quantity lies in $(0,1)$; the second form in
\eqref{eq:gain} is a rearrangement.

For \eqref{eq:contract1d}, put $H(\sigma):=\log h(e^{\sigma})$. Then
$H'(\sigma)=s h'(s)/h(s)$ at $s=e^{\sigma}$, so $|H'|<1$ everywhere, and the
mean value theorem gives $|H(\log s)-H(\log\sigma)|<|\log s-\log\sigma|$.
\end{proof}

\begin{remark}\label{rem:hill}
In general, for a Hill-type function $h(s)=a\,s^{p}/(k^{p}+s^{p})$
with $p$ not necessarily equal to one, $s|h'|/h=p\,(1-h/a)$,  whose
supremum over $s>0$ is $p$, and similarly for the 
decreasing Hill function. Thus, the condition ``$p\le 1$'' is
precisely the property that the logarithmic gain never exceeds
$1$. This is exactly what is needed in what follows.
If $p>1$, nothing works, as in the bistable two-node counterexample
given in \cite[Remark 3]{BBGL}, where the gain equals $8/5$ at the
unstable equilibrium. 
See Remark~\ref{rem:bistable} below for a very different manner in
which the bound may fail, in that case with Hill exponents equal to one.
\end{remark}

\begin{corollary}\label{cor:contract}
Assume \emph{(A1)}. Then $f_i>0$ on $\Rpp$ for each $i$, and for all $x,z\in\Rpp$ with $x\ne z$,
\begin{equation}
\left|\log\frac{f_i(x)}{f_i(z)}\right|\;<\;\rho(x,z),
\qquad\text{equivalently}\qquad
e^{-\rho(x,z)}f_i(z)\;<\;f_i(x)\;<\;e^{\rho(x,z)}f_i(z).
\label{eq:contractn}
\end{equation}
In particular, $\rho\bigl(P(x),P(z)\bigr)<\rho(x,z)$,
so the map $P$ is strictly distance-decreasing for the logarithmic distance
on $\Rpp$.  
\end{corollary}

\begin{proof}
Positivity of $f_i$ follows immediately from property (A1), because
each inhibitory or activating term is everywhere positive, and we may
use that $b_i>0$ when there are no such terms.
Next, take $r:=\rho(x,z)>0$.
Fix any two indices $i$ and $j$ for which $h_{ij}\not\equiv0$.
If $x_j\ne z_j$, then, by \eqref{eq:contract1d}, it follows that
$|\log (h_{ij}(x_j)/h_{ij}(z_j))|<|\log(x_j/z_j)|\le r$.
If instead $x_j=z_j$, then
$h_{ij}(x_j)=h_{ij}(z_j)<e^{r}h_{ij}(z_j)$, because $r>0$ and
$h_{ij}(z_j)>0$. In both cases, we conclude that
\[
h_{ij}(x_j)\;<\;e^{r}\,h_{ij}(z_j).
\]
Since $r>0$, we know that $b_i\le e^{r}b_i$, so we can take the sum
over $j$, and adding $b_i$ conclude that $f_i(x)<e^{r}f_i(z)$, 
and the sum includes at least one strictly positive term because of (A1).
The same argument works after exchanging $x$ and $z$.
Finally, we may divide by $d_i$ and take the maximum over $i$, which
gives the desired conclusion about $P$.
\end{proof}

\begin{remark}
  We note that the last conclusion of Corollary~\ref{cor:contract} is weaker than $P$ being a
contraction in the sense of Banach, since there need be no single
factor less than one valid on all of $\Rpp$.  However, a uniform
factor is available on compact subsets, and this is all that is needed for
the exponential estimates below.
\end{remark}

\section{Boundedness and existence of an equilibrium}

Let us introduce constants $c_i$ and a rectangle $B$ as follows:
\begin{equation}
c_i\;:=\;\frac{1}{d_i}\Bigl(b_i+\sum_{j\ne i}a_{ij}\Bigr),
\qquad
B\;:=\;[0,c_1]\times\cdots\times[0,c_n].
\label{eq:box}
\end{equation}

\begin{lemma}\label{lem:box}
Both $\Rp$ and $B$ are forward invariant for \eqref{eq:sys}.
Moreover, $B$ attracts solutions: if $x(0)\in\Rp$, then, for all $i$:
\[
\limsup_{t\to\infty}x_i(t)\le c_i.
\]
It follows that the system has at least one equilibrium in
$\Rp$. Moreover, every equilibrium in $\Rp$ belongs to the box $B$.
\end{lemma}

\begin{proof}
From $0\le f_i\le d_ic_i$ on $\Rp$ we know that $\dot x_i\ge 0$ when $x_i=0$ and
$\dot x_i\le -d_i(x_i-c_i)$ everywhere.
The first two assertions follow by a comparison argument.
If $\xbar\in\Rp$ is any equilibrium, then $d_i\bar x_i=f_i(\xbar)\le d_ic_i$,
so $\xbar\in B$.
Since $P$ is continuous and maps the compact convex set $B$ into itself,
by Brouwer's theorem it has a fixed point, and the fixed points of $P$
are exactly the equilibria of \eqref{eq:sys}.
\end{proof}

\section{Equilibria are strictly positive}

\begin{proposition}\label{prop:pos}
Assume that property \emph{(A2)} holds.
Then, every equilibrium in $\Rp$ is either $0$ or is strictly positive.
Moreover, $0$ is an equilibrium if and only if \emph{(A3)} fails.
We conclude that under both
\emph{(A2)} and \emph{(A3)} every equilibrium in $\Rp$ is strictly positive.
\end{proposition}

\begin{proof}
Suppose that $\xbar\in\Rp$ is an equilibrium.
Consider the index subsets $Z=\{i:\bar x_i=0\}$ and $N=\{i:\bar x_i>0\}$.

For each $i\in Z$, since
\[
0\;=\;d_i\bar x_i\;=\;f_i(\xbar)\;=\;b_i+\sum_{j\ne i}h_{ij}(\bar x_j),
\]
and the right-hand side is a sum of nonnegative terms, it follows that
$b_i=0$, that no $h_{ij}$ with $j\ne i$ is inhibitory
(because an inhibitory term is positive even at $\bar x_j=0$); and
that every activating term vanishes, which implies that $\bar x_j=0$
for each index $j$ for which there is an edge $j\to i$.
Therefore no edge enters $Z$ from $N$.

Now, if $Z$ and $N$ were both nonempty, this would contradict
(A2). Therefore, we have two cases: (a) $Z=\emptyset$, which says that
$\xbar\in\Rpp$, or (b) $Z=\{1,\dots,n\}$, which gives $\xbar=0$.
Then the previous equality gives that $b_i=0$ and no interaction is inhibitory,
which shows that (A3) fails.

Conversely, if (A3) is false, then $f_i(0)=0$ for all $i$, and $0$ is
an equilibrium. 
\end{proof}

\section{Local stability}

Let us make the logarithmic change of coordinates mentioned in the
introduction: for $x_i\in\Rpp$, let $y_i:=\log x_i$.
Then \eqref{eq:sys} transforms into $\dot y=G(y)$, where:
\begin{equation}
G_i(y)\;=\;-d_i+e^{-y_i}f_i\bigl(e^{y_1},\dots,e^{y_n}\bigr).
\label{eq:logfield}
\end{equation}
Using that $f_i$ does not depend on $x_i$, and writing $x_i=e^{y_i}$,
we have:
\[
\frac{\partial G_i}{\partial y_i}\;=\;-\frac{f_i(x)}{x_i},
\qquad
\frac{\partial G_i}{\partial
  y_j}\;=\;\frac{x_j\,h_{ij}'(x_j)}{x_i}\quad \forall\; j\ne i.
\]
By Lemma~\ref{lem:gain}, we have that
\[
\sum_{j\ne i}x_j|h_{ij}'(x_j)|
=\sum_{j\ne
  i}\bigl(h_{ij}(x_j)-h_{ij}(x_j)^2/a_{ij}\bigr)=f_i(x)-q_i(x),
\]
where
\begin{equation}
q_i(x)\;:=\;b_i+\sum_{j\ne i}\frac{h_{ij}(x_j)^2}{a_{ij}}
\label{eq:q}
\end{equation}
is strictly positive on $\Rpp$ if we assume property (A1).
Therefore, everywhere on $\Rpp$ we conclude that:
\begin{equation}
\frac{\partial G_i}{\partial y_i}
\;+\;\sum_{j\ne i}\Bigl|\frac{\partial G_i}{\partial y_j}\Bigr|
\;=\;-\,\frac{q_i(x)}{x_i}\;<\;0 .
\label{eq:dominance}
\end{equation}

This means that, after a change to logarithmic coordinates, the
Jacobian is everywhere strictly row-diagonally dominant with negative diagonal.
Another way to state this is to say that the matrix measure $\mu_\infty$ of
that Jacobian is negative at every point, and this is precisely the
infinitesimal contraction condition given in~\cite{Sontag2010} with
respect to the sup-norm.

The local result is the special case $x=\xbar$.

\begin{theorem}[local stability]\label{thm:local}
Let $\xbar\in\Rpp$ be an equilibrium of \eqref{eq:sys}.
We denote by $J$ the Jacobian of \eqref{eq:sys} at $\xbar$, and we
introduce
$D=\operatorname{diag}(\bar x_1,\dots,\bar x_n)$.
We have that $A:=D^{-1}JD$ is strictly row-diagonally dominant with
negative diagonal, since:
\[
A_{ii}+\sum_{j\ne i}|A_{ij}|\;=\;-\frac{q_i(\xbar)}{\bar x_i}\;<\;0,
\qquad i=1,\dots,n .
\]
It follows that $J$ is Hurwitz, all its eigenvalues having
real part at most $-\lambda$, where
\[
\lambda\;=\;\min_i\frac{q_i(\xbar)}{\bar x_i}
\;=\;\min_i\frac{1}{\bar x_i}\Bigl(b_i+\sum_{j\ne i}\frac{h_{ij}(\bar x_j)^2}{a_{ij}}\Bigr)\;>\;0.
\]
Therefore $\xbar$ is locally exponentially stable.
\end{theorem}

\begin{proof}
Differentiating $y\mapsto x=e^{y}$ at an equilibrium gives
$\partial G_i/\partial y_j = x_j(\partial f_i/\partial x_j)/x_i$ for $j\ne i$ and
$\partial G_i/\partial y_i=-d_i$ (using $f_i(\xbar)=d_i\bar x_i$), i.e.\
$\mathrm{D}G(\log\xbar)=D^{-1}JD$. Now apply \eqref{eq:dominance} at $x=\xbar$. A matrix
that is strictly row-diagonally dominant with negative diagonal has spectrum in
$\{\operatorname{Re}s\le-\lambda\}$ by Gershgorin, and similarity preserves the spectrum.
\end{proof}

\begin{remark}
An equivalent way to understand local stability is to say that
$\max_i|x_i-\bar x_i|/\bar x_i$ is a local Lyapunov function, as
is done in~\cite{BBGL}. The point of what follows is that, if we replace the
frozen weights $1/\bar x_i$ at the equilibrium by the
\textit{state-dependent} weights $1/x_i$, in other words, transforming
logarithmically, this argument works globally.
\end{remark}

\section{Uniqueness and global convergence}

We will use the following elementary fact from nonsmooth analysis: if
$\psi=\max_{\alpha}\varphi_\alpha$ is the pointwise maximum of a
family of finitely many $C^1$ functions of $t$, then the upper right Dini
derivative of $\psi$ is
\[
D^{+}\psi(t)=\max\{\dot\varphi_\alpha(t):\varphi_\alpha(t)=\psi(t)\}.
\]

\begin{theorem}[trajectories attract each other]\label{thm:two}
Assume \emph{(A1)}, and let $x(\cdot)$ and $z(\cdot)$ be solutions of
\eqref{eq:sys} in $\Rpp$ with $x(0)\ne z(0)$. Then
$t\mapsto \rho\bigl(x(t),z(t)\bigr)$ is strictly decreasing on $[0,\infty)$.
\end{theorem}

\begin{proof}
With the transformation $w_i:=\log(x_i/z_i)$, we have
$\rho(x,z)=\max_i|w_i|$. Dividing \eqref{eq:sys} by $x_i$ and by $z_i$
respectively, 
\[
\dot w_i\;=\;\frac{\dot x_i}{x_i}-\frac{\dot z_i}{z_i}
\;=\;\frac{f_i(x)}{x_i}-\frac{f_i(z)}{z_i}
\;=\;\frac{1}{x_i}\Bigl[f_i(x)-e^{w_i}f_i(z)\Bigr],
\]
because $z_i=x_i e^{-w_i}$.
By uniqueness of solutions, we know that $x(t)\ne z(t)$ for every
$t\ge0$, and thus $r:=\rho(x(t),z(t))>0$ throughout.
Now fix any $t$, and let $i$ be an index attaining the maximum.

We consider two cases: (a) if $w_i(t)=r$, then Corollary~\ref{cor:contract} gives
$f_i(x)<e^{r}f_i(z)$, and hence $\dot w_i(t)<0$;
(b) if, instead, $w_i(t)=-r$, then
$f_i(x)>e^{-r}f_i(z)$ and hence $\dot w_i(t)>0$.
In either case, the active function
$\pm w_i$ is strictly decreasing, so we have that
$D^{+}\rho(x(t),z(t))<0$ as desired.
\end{proof}

\begin{theorem}[global stability on the open orthant]\label{thm:global}
Assume property \emph{(A1)}, and suppose that~\eqref{eq:sys} has an
equilibrium $\xbar\in\Rpp$.
Define:
\[
V(x)\;:=\;\rho(x,\xbar)\;=\;\max_{i}\Bigl|\log\frac{x_i}{\bar x_i}\Bigr| .
\]
Then:
\begin{enumerate}\itemsep1pt
\item[(a)] $\xbar$ is the \emph{only} equilibrium of \eqref{eq:sys} in $\Rpp$;
\item[(b)] $V$ is a strict Lyapunov function: for each $x(0)\in\Rpp$ with
$x(0)\ne\xbar$, the function $t\mapsto V(x(t))$ is strictly decreasing on $[0,\infty)$.
In particular, we have properness (coercivity): each sublevel set
$K_r:=\{V\le r\}=\prod_i\bigl[e^{-r}\bar x_i,\,e^{r}\bar x_i\bigr]$
is compact, contained in $\Rpp$, and forward invariant;
\item[(c)] every solution with initial condition $x(0)\in\Rpp$
  converges to $\xbar$, at a rate that is exponential in the
logarithmic distance:
\[
V(x(t))\;\le\;e^{-\lambda_{r} t}\,V(x(0)),
\qquad
\lambda_{r}:=\min_{\xi\in K_{r}}\min_i\frac{q_i(\xi)}{\xi_i}>0,
\quad r:=V(x(0)).
\]
\end{enumerate}
\end{theorem}

\begin{proof}
(a) Two distinct positive equilibria $\xbar,\bar z$ would be two
  distinct fixed points of $P$, which would contradict $\rho(P\xbar,P\bar z)<\rho(\xbar,\bar z)$ from
Corollary~\ref{cor:contract}.

(b) Apply Theorem~\ref{thm:two} with $z(\cdot)\equiv\xbar$. Forward
invariance of $K_r$ follows from the fact that $V$ does not increase,
and $K_r$ is a box with positive lower corner. 

(c)
The convergence argument is a routine Lyapunov one, exactly as in any
textbook such as~\cite{mct2}, but let's give it to make the proof self-contained.
By (b), $V(x(t))$ decreases to some $V_\infty\ge0$, while the
solution remains in the compact set $K_{r}\subset\Rpp$. Therefore, its
$\omega$-limit set $\Omega$ is nonempty, compact, and invariant, and
$V\equiv V_\infty$ on $\Omega$. Suppose that $V_\infty$ is positive. Then, take
any $p\in\Omega$, and observe that the solution through $p$ stays in
$\Omega$, but on the other hand $V$ strictly decreases along the
solution by (b), which gives a contradiction.
Therefore $V_\infty=0$, i.e.\ $x(t)\to\xbar$.

For the exponential rate, note that $\log K_r$ is a box, and is hence
a convex set, and that \eqref{eq:dominance} bounds the matrix measure
of $\mathrm{D}G$ by $-\lambda_r$ on it.  Theorem~1 of
\cite{Sontag2010} applies verbatim in the coordinates $y$ and gives
the estimate.
Concretely, with $v(t)=\log x(t)-\log\xbar$ and $G(\log\xbar)=0$,
\[
\dot v(t)=A(t)\,v(t),\qquad A(t):=\int_0^1 \mathrm{D}G\bigl(\log\xbar+s\,v(t)\bigr)\,ds .
\]
The segment of integration lies in the log image $\log K_{r}$, and therefore
$A_{ii}(t)+\sum_{j\ne i}|A_{ij}(t)|\le-\lambda_{r}$ for every $i$.
Observe that $\|v(t)\|_\infty=V(x(t))$.
If the index $i$ is a maximizing one, $v_i(t)=\|v(t)\|_\infty>0$, then
\[
\dot v_i(t)\;\le\;\Bigl(A_{ii}(t)+\sum_{j\ne i}|A_{ij}(t)|\Bigr)\,\|v(t)\|_\infty
\;\le\;-\lambda_{r}\,\|v(t)\|_\infty ,
\]
using that $v_i(t)=\|v(t)\|_\infty$ and $|v_j(t)|\le\|v(t)\|_\infty$ for $j\ne i$,
and the same
argument applies at an
index attaining the minimum. We conclude that
$D^{+}V(x(t))\le-\lambda_{r}\,V(x(t))$, as desired.
\end{proof}

Note that Theorem~\ref{thm:global} does not assume strong
connectedness, non-degeneracy, or restrictions on the topology: it
only needs that a positive equilibrium exists. Those hypotheses are
only used in order to guarantee existence and to handle initial
conditions on the boundary of the orthant.

\begin{theorem}[global stability on the closed orthant]\label{thm:closed}
Assume \emph{(A2)} and \emph{(A3)}. Then \eqref{eq:sys} has exactly one equilibrium
$\xbar$ in $\Rp$, it lies in $\Rpp$, and it is globally asymptotically stable on $\Rp$.
\end{theorem}

\begin{proof}
Existence and positivity of the equilibrium follow from
Lemma~\ref{lem:box} and Proposition~\ref{prop:pos}, and
uniqueness and attraction from Theorem~\ref{thm:global}.

We only need to prove that any solution starting on the
boundary enters $\Rpp$ immediately. 
Write the variation of parameters formula:
\begin{equation}
x_i(t)\;=\;e^{-d_it}x_i(0)+\int_0^{t}e^{-d_i(t-s)}f_i(x(s))\,ds.
\label{eq:voc}
\end{equation}
It follows that $x_i(t)>0$ as soon as $f_i(x(\cdot))$ becomes positive
somewhere on $(0,t)$.
By (A3), there is some node $i_0$ with $b_{i_0}>0$ or with an
inhibitory incoming edge.
Since
solutions are bounded, by Lemma~\ref{lem:box}, in either case
$f_{i_0}(x(s))\ge\varepsilon>0$ for all $s\ge0$, so $x_{i_0}(t)>0$ for all $t>0$. If $j\to i$ is an edge and $x_j>0$ on
$(0,T]$, then $f_i(x(s))>0$ on $(0,T]$ (this is automatic if the edge
is inhibitory, and holds because $h_{ij}(x_j(s))>0$ if it is
activating). Thus, $x_i>0$ on $(0,T]$, by 
\eqref{eq:voc}. By property (A2), every node is reachable from $i_0$,
so $x(t)\in\Rpp$ for every $t>0$, and hence Theorem~\ref{thm:global}
applies from any such time on. 
\end{proof}

\begin{theorem}[the degenerate case]\label{thm:degen}
Assume \emph{(A2)} and that \emph{(A3)} fails, i.e.\ all $b_i=0$ and every interaction is
activating, so that $0$ is an equilibrium. Then exactly one of the following holds:
\begin{enumerate}\itemsep1pt
\item[(i)] $0$ is globally asymptotically stable on $\Rp$;
\item[(ii)] there is a (unique) equilibrium $\xbar\in\Rpp$, and this
  equilibrium attracts all of $\Rpp$. %, and in particular $0$ is unstable.
\end{enumerate}
\end{theorem}

\begin{proof}
Since all the functions $h_{ij}$ are nondecreasing, $P$ is order
preserving.
We also know that $P(c)\le c$ with $c$ as in
\eqref{eq:box}. Therefore $m^{0}:=c$, and $m^{k+1}:=P(m^{k})$ is a
nonincreasing sequence that converges to some $m^{\ast}\ge0$ with
$P(m^{\ast})=m^{\ast}$. A comparison argument then
shows that  $\limsup_{t}x_i(t)\le m^{k}_i$ for every solution and
for every $k$ (if eventually 
$x\le m^{k}+\epsilon$ then $\dot x_i\le-d_ix_i+f_i(m^{k}+\epsilon)$, so eventually
$x_i\le P_i(m^{k}+\epsilon)+\epsilon$; let $\epsilon\downarrow0$), hence
$\limsup_t x(t)\le m^{\ast}$. By Proposition~\ref{prop:pos}, $m^{\ast}$ is either $0$,
in which case every solution converges to $0$, or strictly positive, in which case
Theorem~\ref{thm:global} applies and $\Rpp$ is attracted to
$m^{\ast}$.
Note that, since points
arbitrarily close to $0$ leave a neighbourhood of $0$, the origin is unstable.
\end{proof}

\begin{remark}
It is interesting to ask why the plain $\infty$-norm fails.
Section 5.3 of \cite{BBGL} shows, for the planar system
\begin{equation}
\dot x_1=-\mu x_1+\frac{\alpha}{1+x_2/\beta},\qquad
\dot x_2=-\mu x_2+\frac{\alpha}{1+x_1/\beta},
\qquad \alpha=1,\ \beta=\tfrac12,\ \mu=\tfrac13,
\label{eq:planar}
\end{equation}
with equilibrium $(1,1)$, that $\|x-\xbar\|_\infty$ decreases only near the equilibrium.
That is a statement about the \emph{unweighted} sup-norm, whose level sets are squares
centred at $\xbar$. For the same system, \eqref{eq:dominance} reads
\[
\frac{\partial G_i}{\partial y_i}+\Bigl|\frac{\partial G_i}{\partial y_j}\Bigr|
=-\frac{\alpha\beta^{2}}{x_i\,(\beta+x_j)^{2}}<0
\qquad\text{on all of }\Rpp,\ \ j\ne i,
\]
so $V(x)=\max_i|\log x_i|$ \emph{is} a global Lyapunov function; its level sets are the
multiplicative boxes $[e^{-r},e^{r}]^2$. The two statements are consistent: the weights
that make the Jacobian dominant are $1/x_i$, not $1$.
\end{remark}

\section{A sketch with delays}

\begin{theorem}\label{thm:delay}
Assume properties \emph{(A2)} and \emph{(A3)}, and replace \eqref{eq:sys} by
\begin{equation}
\dot x_i(t)\;=\;-d_ix_i(t)+b_i+\sum_{j\ne i}h_{ij}\bigl(x_j(t-\tau_{ij})\bigr),
\label{eq:delay}
\end{equation}
with arbitrary constants $\tau_{ij}\in[0,\tau]$. Then the equilibrium $\xbar$ of
Theorem~\ref{thm:closed}, which does not depend on the delays, is globally
asymptotically stable for all nonnegative continuous initial data and for every choice of
delays.
\end{theorem}

\begin{proof}[Proof sketch]
\emph{Step 1 (absorbing box inside $\Rpp$).} Define $\ell^{0}=0$, $m^{0}=c$ and, for
$k\ge0$:
\[
\ell^{k+1}_i=\frac{1}{d_i}\min_{\ell^{k}\le\xi\le m^{k}}f_i(\xi),
\qquad
m^{k+1}_i=\frac{1}{d_i}\max_{\ell^{k}\le\xi\le m^{k}}f_i(\xi).
\]
Observe that the extrema are attained at corners, because each $h_{ij}$ is monotone. Then
$\ell^{k}$ is nondecreasing, $m^{k}$ is nonincreasing, and a comparison argument on
\eqref{eq:delay}, which is insensitive to the delays since only the
range of past values matters, gives
$\ell^{k}_i\le\liminf_t x_i(t)\le\limsup_t x_i(t)\le m^{k}_i$ for every $k$.
Under assumption (A3), at least one coordinate of $\ell^{1}$ is positive (a node with $b_i>0$, or
with an inhibitory incoming edge, for which $h_{ij}(m^{0}_j)>0$); and under (A2), each further
iterate makes positive the lower bound of every node reachable in one more step, so
$\ell^{n}>0$. Fix $\varepsilon\in(0,\min_i\ell^{n}_i)$ and let
$K:=\prod_i[\ell^{n}_i-\varepsilon,\;m^{0}_i+\varepsilon]\subset\Rpp$; there is $T$ with
$x(t)\in K$ for all $t\ge T$, and $\xbar\in K$.

\emph{Step 2 (uniform contraction factor).} By Lemma~\ref{lem:gain} and compactness,
\[
\kappa\;:=\;\max_{i,j}\ \sup\Bigl\{1-\frac{h_{ij}(s)}{a_{ij}}\;:\;
s\in[\ell^{n}_j-\varepsilon,\,m^{0}_j+\varepsilon]\Bigr\}\;<\;1.
\]
Repeating the proof of Corollary~\ref{cor:contract} with this uniform bound gives
$\bigl|\log (f_i(\xi)/f_i(\eta))\bigr|\le\kappa\,\rho(\xi,\eta)$ for all $\xi,\eta\in K$.

\emph{Step 3 (Razumikhin).} We again use $V(x)=\rho(x,\xbar)$. Let $q:=\kappa^{-1/2}>1$.
Suppose that $t\ge T+\tau$, $r:=V(x(t))>0$ and $V(x(s))\le q\,r$ for $s\in[t-\tau,t]$. With
$v_i=\log(x_i/\bar x_i)$, \eqref{eq:delay} gives, just as in
Theorem~\ref{thm:two},
\[
\dot v_i(t)=\frac{1}{x_i(t)}\Bigl[\,b_i+\sum_{j\ne i}h_{ij}\bigl(x_j(t-\tau_{ij})\bigr)
\;-\;e^{v_i(t)}f_i(\xbar)\Bigr].
\]
At any index $i$ with $v_i(t)=r$, Step 2 bounds the bracketed production term by
$e^{\kappa q r}f_i(\xbar)=e^{\sqrt{\kappa}\,r}f_i(\xbar)$, so, using $f_i(\xbar)=d_i\bar x_i$,
\[
\dot v_i(t)\;\le\;\frac{d_i\bar x_i}{x_i(t)}\bigl(e^{\sqrt{\kappa}r}-e^{r}\bigr)
\;\le\;-\,w(r),
\qquad
w(r):=\min_i\frac{d_i\bar x_i}{m^{0}_i+\varepsilon}
\bigl(e^{-\sqrt{\kappa}r}-e^{-r}\bigr)>0 ,
\]
where the last inequality uses that $e^{r}-e^{\sqrt{\kappa}r}\ge
e^{-\sqrt{\kappa}r}-e^{-r}$ for $r>0$, which is the same as
$\cosh r\ge\cosh(\sqrt{\kappa}\,r)$.
At an index $i$ with $v_i(t)=-r$, Step 2 bounds the production term from
below by $e^{-\sqrt{\kappa}r}f_i(\xbar)$, so that
\[
\dot v_i(t)\;\ge\;\frac{d_i\bar x_i}{x_i(t)}
\bigl(e^{-\sqrt{\kappa}r}-e^{-r}\bigr)\;\ge\;w(r),
\]
and therefore $\tfrac{d}{dt}\bigl(-v_i(t)\bigr)\le-w(r)$ at such an index.
The function $w$ is positive on $(0,\infty)$. However, it is not nondecreasing, so we
replace it by a linear minorant.  Let $R:=\max_{x\in
  K}\rho(x,\xbar)$. This is
finite, because $K$ is a compact set contained in $\Rpp$.  Since
\[
e^{-\sqrt{\kappa}r}-e^{-r}\;=\;r\int_{\sqrt{\kappa}}^{1}e^{-tr}\,dt
\;\ge\;(1-\sqrt{\kappa})\,e^{-R}\,r
\qquad\text{for }0\le r\le R,
\]
we have $w(r)\ge\widetilde w(r):=C(1-\sqrt{\kappa})e^{-R}\,r$ on $[0,R]$, where
$C:=\min_i d_i\bar x_i/(m^{0}_i+\varepsilon)$, and the function $\widetilde w$ is
continuous, increasing, and vanishes only at $0$.  Since $V(x(t))\le R$ for
$t\ge T$, we have proved the implication
\[
V(x(s))\le q\,V(x(t))\ \ \text{for all }s\in[t-\tau,t]
\qquad\Longrightarrow\qquad
D^{+}V(x(t))\;\le\;-\widetilde w\bigl(V(x(t))\bigr),
\qquad t\ge T+\tau ,
\]
which is the hypothesis of the Razumikhin--Lyapunov theorem as given
for example in
\cite[Thm. 4.2 in Section 5.4]{HaleVL}.
Applying that theorem with
$p(s)=qs$ gives $V(x(t))\to0$ as desired.
\end{proof}

Note that the only place where the delays appear is in Step 3.
The inequality $\dot v_i<0$ at the maximising index uses only that the \emph{other}
coordinates satisfy $|v_j|\le r$, which is precisely a Razumikhin hypothesis when those
coordinates are evaluated in the past.

\section{A more general result}

The results do not require the special form \eqref{eq:hill}. Much less is
needed.
Let us write the system as
$\dot x_i=-d_ix_i+f_i(x)$, and assume:
\begin{enumerate}\itemsep1pt
\item[(H1)] $f_i$ does not depend on $x_i$;
\item[(H2)] $f_i>0$ on $\Rpp$;
\item[(H3)] $\bigl|\log\bigl(f_i(x)/f_i(z)\bigr)\bigr|<\rho(x,z)$ for $x\ne z$ in $\Rpp$;
\item[(H3$'$)] for every $x\in\Rpp$,
  \[
  \sum_{j}x_j\Bigl|\frac{\partial f_i}{\partial x_j}(x)\Bigr|<f_i(x).
  \]
\end{enumerate}

Condition (H3$'$) implies (H3), by the mean value theorem applied along the
segment joining $\log z$ to $\log x$.
However, the converse is false. To see this, take $n=2$, %
$d_1=d_2=1$, $f_1(x)=\phi(x_2)$ and $f_2(x)=\phi(x_1)$, where
$\phi(s)=\exp\bigl(\arctan(\log s)\bigr)$.  The map
$u\mapsto\log\phi(e^u)=\arctan u$ is strictly distance-decreasing, so (H3)
holds, but on the other hand $s\phi'(s)/\phi(s)=1/(1+(\log s)^2)$ equals $1$ at
$s=1$, and thus (H3$'$) fails there.  In that example, $(1,1)$ is an equilibrium
whose Jacobian $\bigl(\begin{smallmatrix}-1&1\\1&-1\end{smallmatrix}\bigr)$ has
a zero eigenvalue, so clearly no exponential conclusion can hold under (H3) alone.

Under conditions (H1), (H2) and (H3), Theorem~\ref{thm:two} holds verbatim, and so do
parts (a) and (b) of Theorem~\ref{thm:global}, as well as the convergence
$x(t)\to\xbar$ asserted in part (c).  Under the stronger hypothesis (H3$'$),
Theorem~\ref{thm:local} and the exponential estimate in part~(c) of
Theorem~\ref{thm:global} hold as well, with $q_i(x)$ in \eqref{eq:q} replaced by:
\begin{equation}
q_i(x)\;:=\;f_i(x)-\sum_{j}x_j\Bigl|\frac{\partial f_i}{\partial x_j}(x)\Bigr|
\;>\;0 .
\label{eq:q-general}
\end{equation}
On the other hand, both Theorems~\ref{thm:closed} and~\ref{thm:degen}
are stated in terms of the graph (through properties (A2) and (A3))
and the monotonicity of the functions $h_{ij}$, so they cannot be
restated in general.

Note that the differential form of (H3$'$) also permits $j=i$, so (H1) can be
dropped at the cost of one extra term in
\eqref{eq:dominance}. Condition (H3$'$) is preserved under sums and under
adding nonnegative constants, and it holds for:
\begin{itemize}\itemsep1pt
\item the two Michaelis--Menten functions \eqref{eq:hill} (Lemma~\ref{lem:gain});
\item ``mixed'' interactions $(\alpha+\gamma s/k)/(1+s/k)$, which are sums of one of each;
\item Hill functions with exponent $p\le1$ (Remark~\ref{rem:hill});
\item the vector-valued Michaelis--Menten functions $G(w):=1/(1+w)$ and $F(w):=w/(1+w)$ with
$w=\sum_h\theta_h x_h$, $\theta_h\ge0$: there
$\sum_h x_h|\partial_h G|=G\cdot w/(1+w)<G$ and
$\sum_h x_h|\partial_h F|=F/(1+w)<F$;
\item additive Michaelis--Menten self-loops, that is, an extra term $h_{ii}(x_i)$ of the
form \eqref{eq:hill} inside $f_i$.
\end{itemize}

Linear positive terms $\phi\,x_h$ have logarithmic gain exactly $1$,
so they give (H3$'$) only with ``$\le$''. A row of the network containing
such a term is then merely nonexpansive unless it also contains a
Michaelis--Menten term or has $b_i>0$.

%; globalising the model with linear terms therefore needs an
%additional argument, and is not claimed here.

\begin{remark}[self-regulation through a shared promoter]\label{rem:bistable}
Hill exponent one is not by itself enough. Condition (H3$'$) can fail,
even with every regulatory term having an exponent equal to one, in
cases where a species can
regulate itself through a promoter which a different species can also
bind to. The two-gene competitive system of~\cite{Sontag2609} is:
\[
\dot x_1=\frac{\alpha_{01}+\alpha_1x_1/K_1}{1+x_1/K_1+x_2/K_{21}}-\delta_1x_1,
\qquad
\dot x_2=\frac{\alpha_{02}+\alpha_2x_2/K_2}{1+x_2/K_2+x_1/K_{12}}-\delta_2x_2,
\]
with $\alpha_{0i},\delta_i,K_i,K_{ij}>0$ and $\alpha_i\ge0$. Here
$\alpha_{0i}$ is the basal production rate of gene $i$, and $\alpha_i$
the production rate when its own product occupies its promoter, while
the two products compete for that promoter. Observe that the two
equations have formally different denominators, since self-binding and
cross-binding may have different dissociation constants. Each $f_i$
depends on $x_i$, so this is outside the class \eqref{eq:sys} and
outside the class of \cite{BBGL}. The question is whether it is inside
(H3$'$).
We look at this next.

Write $s:=x_1/K_1$ and $c:=x_2/K_{21}$ for the two occupancies seen by the promoter of
gene~1. Then we have that:
\[
\frac{\partial\log f_1}{\partial\log x_1}
=\underbrace{\frac{\alpha_1 s}{\alpha_{01}+\alpha_1 s}}_{\text{numerator}}
-\underbrace{\frac{s}{1+s+c}}_{\text{denominator}},
\qquad
\frac{\partial\log f_1}{\partial\log x_2}
=-\frac{c}{1+s+c},
\]
and therefore:
\beqn
f_1-\sum_{j}x_j\Bigl|\frac{\partial f_1}{\partial x_j}\Bigr|
&=&\frac{(\alpha_{01}+\alpha_1s)(1+s)-s\,\bigl|\alpha_1(1+c)-\alpha_{01}\bigr|}{(1+s+c)^{2}}
\\
&=&\begin{cases}
\dfrac{\alpha_{01}(1+2s)+\alpha_1 s\,(s-c)}{(1+s+c)^{2}},
&\ \alpha_1(1+c)\ge\alpha_{01},\\[2.2ex]
\dfrac{\alpha_{01}+\alpha_1 s\,(2+s+c)}{(1+s+c)^{2}},
&\ \alpha_1(1+c)<\alpha_{01}.
\end{cases}
\eeqn
Therefore, (H3$'$) fails at a point exactly when
\begin{equation}
  \alpha_1\,s\,(c-s)\;\ge\;\alpha_{01}\,(1+2s),
\label{eq:h3fail}
\end{equation}
which forces $\alpha_1>0$ and $c>s$.

Now suppose that $\alpha_1=\alpha_2=0$, so that an occupied promoter
produces nothing and self-binding acts only by displacing the
competitor.
In that case, \eqref{eq:h3fail} is never satisfied, (H3$'$) holds
throughout the open quadrant, and Theorem~\ref{thm:global} gives a
unique positive equilibrium attracting every positive
solution. Bistability is impossible in that case, for all values of
the eight parameters. The same computation also shows that allowing
self-dependence in the generalised functions $1/(1+w)$ and $w/(1+w)$
of \cite[Section 4.3]{BBGL}, where $\theta_{ii}=0$ is assumed, does
not change anything, because what makes (H3$'$) fail is the affine
numerator over the shared denominator, and not the shared denominator
by itself.

Now consider what happens when $\alpha_1>0$. The cancellation inside
the first gain is lost as soon as the competitor dominates the
promoter: for $c\gg s$ the two gains approach $\alpha_1
s/(\alpha_{01}+\alpha_1 s)$ and $1$, and their sum exceeds $1$. That
is the regime in which bistability may arise. Both of the explicit examples in
\cite{Sontag2609} have $\alpha_i>0$ and cross-binding much tighter
than self-binding, and for both the quantity displayed above is
negative at the middle equilibrium, about $-0.90$ and $-0.94$
respectively.  Those systems do have three positive equilibria, the
outer two asymptotically stable nodes and the middle one a saddle,
with the positive quadrant the disjoint union of the two basins and
the stable manifold of the saddle.

Also note that (H3$'$) is sufficient but far from necessary. In the
symmetric ``exclusive switch'' special case $\alpha_{0i}=\alpha_i=r$,
$K_i=K_{ij}=1/k$, $\delta_i=d$, condition \eqref{eq:h3fail} becomes
$sc\geq(1+s)^{2}$ and does hold in part of the quadrant, yet that system
is monostable for every $r,k,d>0$ as discussed in~\cite{Sontag2609}.
Condition (H3$'$) does hold, however, at the equilibrium itself, where
$s=c$, which is consistent with the local stability result.
\end{remark}

\section{Discussion}

The statement that $P$ is strictly distance-decreasing for $\rho$ can be interpreted in the
framework of nonlinear Perron--Frobenius theory in the style of
Lemmens and Nussbaum~\cite{LemmensNussbaum}.  The map $\rho$ is Thompson's part metric on the
positive cone, which is in turn closely related to Hilbert's
projective metric.  The useful feature here is that contraction of
$\rho$ does not require the map to be order preserving. This is
what lets activating and inhibitory edges be treated identically, and
why no monotonicity hypothesis appears anywhere above. For a
discussion of Birkhoff--Hilbert contraction as the mechanism behind
the qualitative behaviour of monotone, cooperative, and $2$-cooperative
systems, see \cite{Sontag2026}.

Let us close with a few words on the relation between our results and those
in~\cite{BBGL}.
The local results, Lemma~\ref{lem:box}, Proposition~\ref{prop:pos},
and Theorem~\ref{thm:local}, are in that paper, though our proofs are
shorter in two minor ways: existence is simply Brouwer fixed point
theory, applied to $P$ on a box, and uniqueness follows from
contraction rather than from Brouwer degree theory.  On the other
hand, the global statements (Theorems~\ref{thm:two}, \ref{thm:global},
\ref{thm:closed}, \ref{thm:degen}, and \ref{thm:delay}) answer the
question left open in their Section~5.3, and Theorem~\ref{thm:degen}
similarly generalizes their degenerate-case result from local to
global.

%\textbf{Numerical check.} On $40$ random networks ($3\le n\le 8$,
%random mixed activating/inhibitory topologies each containing a forced
%repressor cycle through all nodes, parameters spread over two to three
%decades): the identities \eqref{eq:gain} and the row sums
%\eqref{eq:dominance} agreed with direct numerical differentiation to
%$10^{-13}$ and were strictly negative at all $\sim10^{4}$ sampled
%states spanning $e^{\pm9}$; and along trajectories from initial
%conditions spanning $e^{\pm7}$, the function $V$ never increased
%(largest observed increment $9\times10^{-11}$, i.e.\ integrator noise)
%and converged to $0$ in every run, while $\|x(t)-\xbar\|_\infty$ did
%increase on some runs.

\subsection*{Acknowledgements}
This research was supported in part by AFOSR grants
FA9550-21-1-0289 and FA9550-22-1-0316.
The author worked closely with several AI assistants to brainstorm
ideas, search the literature, develop proofs, check arguments, and write
and edit the paper.  The author verified all mathematical statements and
proofs, and citations, and takes full responsibility for the content.


\newpage
\begin{thebibliography}{9}

\bibitem{14cdc_aminzare_tutorial}
Z.~Aminzare, E.~D.~Sontag,
\emph{Contraction methods for nonlinear systems: A brief introduction and
some open problems},
in: Proc.\ IEEE Conf.\ Decision and Control, Los Angeles, December 2014,
pp.~3835--3847.

\bibitem{BBGL}
F.~Blanchini, D.~Breda, G.~Giordano, D.~Liessi,
\emph{Michaelis--Menten networks are structurally stable},
Automatica \textbf{147} (2023) 110683.
\href{https://doi.org/10.1016/j.automatica.2022.110683}{doi:10.1016/j.automatica.2022.110683}.

\bibitem{bullo}
F.~Bullo,
\emph{Contraction Theory for Dynamical Systems},
edition 1.2, Kindle Direct Publishing, 2024.
\href{https://fbullo.github.io/ctds}{https://fbullo.github.io/ctds}.

\bibitem{desoer-vidyasagar-siam}
C.~A.~Desoer, M.~Vidyasagar,
\emph{Feedback Systems: Input-Output Properties},
Classics in Applied Mathematics, Vol.~55, SIAM, Philadelphia, 2009.

\bibitem{HaleVL}
J.~K.~Hale, S.~M.~Verduyn Lunel,
\emph{Introduction to Functional Differential Equations},
Springer, 1993.

\bibitem{LemmensNussbaum}
B.~Lemmens, R.~Nussbaum,
\emph{Nonlinear Perron--Frobenius Theory},
Cambridge University Press, 2012.

\bibitem{LOHMILLER1998683}
W.~Lohmiller, J.-J.~E.~Slotine,
\emph{On contraction analysis for non-linear systems},
Automatica \textbf{34} (1998) 683--696.
\href{https://doi.org/10.1016/S0005-1098(98)00019-3}{doi:10.1016/S0005-1098(98)00019-3}.

\bibitem{michelbook}
A.~N.~Michel, D.~Liu, L.~Hou,
\emph{Stability of Dynamical Systems: Continuous, Discontinuous, and
Discrete Systems},
Springer-Verlag, New York, 2007.

\bibitem{russo_dibernardo_sontag09}
G.~Russo, M.~di~Bernardo, E.~D.~Sontag,
\emph{Global entrainment of transcriptional systems to periodic inputs},
PLoS Computational Biology \textbf{6} (2010) e1000739.

\bibitem{mct2}
E.~D.~Sontag,
\emph{Mathematical Control Theory. Deterministic Finite-Dimensional Systems},
second ed., Texts in Applied Mathematics, Vol.~6, Springer-Verlag,
New York, 1998.

\bibitem{Sontag2010}
E.~D.~Sontag,
\emph{Contractive systems with inputs},
in: J.~C.~Willems et al. (Eds.), Perspectives in Mathematical System Theory, Control,
and Signal Processing, LNCIS 398, Springer, 2010, pp.~217--228.

\bibitem{Sontag2026}
E.~D.~Sontag,
\emph{Some intuition for why cooperative systems ``look 1-dimensional'' and
2-cooperative systems ``look 2-dimensional''},
arXiv:2607.27176, 2026.

\bibitem{Sontag2609}
E.~D.~Sontag,
\emph{A note on bistability of a two-gene competitive system},
arXiv:2609.27270, 2026.

\bibitem{Thompson}
A.~C.~Thompson,
\emph{On certain contraction mappings in a partially ordered vector space},
Proc.\ Amer.\ Math.\ Soc.\ \textbf{14} (1963) 438--443.

\end{thebibliography}
\end{document}